\documentclass[a4paper]{article}
\usepackage[utf8]{inputenc}
\usepackage{amstext,amsmath,amsthm,amssymb}   
\usepackage{breqn}
\usepackage{latexsym}
\usepackage{graphicx}
\usepackage{url}

\newtheorem{theo}{Theorem}
\newtheorem{defn}{Definition}
\newtheorem{prop}{Proposition}
\newtheorem{lemma}{Lemma}
\newtheorem{rem}{Remark}
\newtheorem{cor}{Corollary}

\newcommand{\N}{\mathbb{N}}
\newcommand{\Z}{\mathbb{Z}}

\newcommand{\nrela}[2][X]{\;\hat{r}^{#2,N}_{#1}\;}
\newcommand{\nrelaeq}[2][X]{\;r^{#2,N}_{#1}\;}
\newcommand{\parp}[3][X]{{\cal R}^{#2, #3}_{#1}}
\newcommand{\Nlan}[2][N]{\mathcal{L}_{#1}(#2)}
\newcommand{\lan}[1]{\mathcal{L}(#1)}
\newcommand{\ngra}[3][X]{\mathcal{G}^{#2,#3}_{#1}}

\newcommand{\alphab}[1]{\mathcal{#1}}

\newcommand{\nsgra}[3]{\mathcal{G}_{#1, #2, #3}}

\newcommand\ie{\textit{i.e} }

\newcommand{\sm}{\sqsubset}
\newcommand{\defi}[1]{\emph{#1}}
\newcommand{\slr}[2]{S_{#1}({#2})}
\newcommand{\sll}[2]{P_{#1}({#2})}

\newcommand\card[1]{\left|#1\right|}

\newcommand\words[1]{#1^{\infty}}

\newcommand{\condi}[0]{\mid}
\newcommand{\projI}[1][N]{\varphi^{#1}_{X}}
\newcommand{\proj}[2][N]{\varphi^{#2,#1}_{X}}
\newcommand{\projBis}[2][N]{\varphi^{#2,#1}_{Y}}
\newcommand{\projXI}[1][N]{\pi^{#1}_{X}}
\newcommand{\projX}[2][N]{\pi^{#2,#1}_{X}}

\newcommand{\projZ}[0]{\delta}
\newcommand{\NBP}[2]{{#1}^{[#2]}}
\newcommand{\x}[1]{{c}^{(#1)}}
\newcommand{\y}[1]{{s}^{(#1)}}
\newcommand{\w}[1]{{t}^{(#1)}}
\newcommand{\z}[1]{{\ell}^{(#1)}}
\newcommand{\s}[1]{{k}^{(#1)}}
\newcommand{\similar}[0]{\sim}
\newcommand{\lforb}[1]{L(#1)}
\newcommand{\Stp}[0]{L}
\newcommand{\dix}[0]{\sharp}
\newcommand{\toalph}[1]{\left[#1\right]}

\newcommand{\colt}[3]{\scriptsize\left[\hspace{-.75em}\begin{array}{c}#1\\#2\\#3\end{array}\hspace{-.75em}\right]}

\title{$N$-block presentations and decidability of direct conjugacy between Subshifts of Finite Type}
\author{\'Emilie Delnieppe\\ \small Aix-Marseille Universit\'e, CNRS, Centrale Marseille, I2M UMR7373, Marseille, France\\ 
\small E-mail: \url{emilie.delnieppe@univ-amu.fr}}

\begin{document}

\maketitle

\begin{abstract}
We consider the problem of inverting the transformation which consists in replacing a word by the sequence of its blocks of length $N$, i.e. its so-called $N$-block presentation.
It was previously shown that among all the possible preimages of an $N$-block presentation, there exists a particular one which is maximal in the sense that all the other preimages can be obtained from it by letter to letter applications.
We give here a combinatorial characterization of the maximal preimages of $N$-block presentations.
Using this characterization, we show that, being given two subshifts of finite type $X$ and $Y$,  the existence of two numbers $N$ and $M$ such that the  $N$-block presentation of $X$ is similar to the $M$-block presentation of $Y$, which implies that $X$ and $Y$ are conjugate, is decidable. 

{\bf Key words:} Symbolic dynamics, Higher block presentation, Subshift of Finite Type, Conjugacy

\end{abstract}

\section{Introduction}

Sliding-block coding is a central transformation in Symbolic Dynamics, because they represent all possible dynamical factor maps between symbolic systems.
The canonical sliding-block code of length $N$ is the $N^{\mbox{\tiny th}}$ higher-block code, because any sliding-block code is a composition of a higher-block code and a letter-to-letter application \cite{Lind}.
It maps a given word to the sequence of its blocks of length $N$, its so-called $N$-block presentation.
We call \defi{$N$-preimage} of a word $u$ every word of which the $N^{\mbox{\tiny th}}$ higher-block code is equal to $u$, up to a renaming of its letters. A preimage $v$ of $u$ is said \defi{maximal} if all other preimages of $u$ can be obtained through letter-to-letter maps from $v$. The existence of such a preimage for any block presentation was proved in \cite{ecriture}, which also provided a combinatorial characterization of $N$-block presentations.
We study here some properties of preimages of block presentations. In particular, we give a combinatorial characterization of those which are maximal. 

While the conjugacy between two SFT in the one-sided case is decidable
\cite{WIL}, it remains an open question in the two-sided case while being connected to many open problems \cite{BOY}. 
Among related results, let us mention the decidability of
conjugacy for tree-shifts of finite type \cite{BEA}, or of
strong shift equivalence \cite{KIM}. We introduce here the notion of \defi{direct
conjugacy} as follows: two SFTs $X$ and $Y$ are directly conjugate if
two positive integers $M$ and $N$ exist such that $X^{M} = Y^{M}$ up to a renaming of their letters.
By using results obtained about preimages of block presentations, we prove the decidability of direct conjugacy between two SFTs.

The rest of the paper is organized as follows. Section \ref{s:notation} presents the notations and definitions. Since \cite{ecriture} dealed only with single words, we adapt in Section \ref{s:preimage} some of its notions and results about preimages of higher block presentations in order to deal with set of words. In Section \ref{s:composing}, we show that preimages of higher block presentations cannot be composed in the general case but we characterize the situations in which it is possible. Section \ref{s:charac} presents our characterization of maximal preimages of $N$-block presentations. In the last section, we show that the direct conjugacy between SFTs is decidable by using results obtained on preimages in the previous sections.

\section{Notation and definitions}\label{s:notation}

\subsection{Words and word sets}
We put $\card{\mathcal{S}}$ for the cardinal of any finite set $\mathcal{S}$. An alphabet $\alphab{A}$ is a finite set of elements called letters or symbols. A word on an alphabet $\alphab{A}$ is a finite, infinite or bi-infinite sequence of symbols in $\alphab{A}$. Each time it will need to be specified, we will talk about finite, infinite or bi-infinite words. Infinite (resp. bi-infinite) words are indiced on $\N$ (resp. $\Z$). We put $|w|$ for the length of the finite words $w$ which are indiced from $0$, i.e. $w = w_{0} \dots w_{|w|-1}$.
For two positions $i\leq j$ of $w$, $w_{[i,j]}$ denotes the subword of $w$ which starts at position $i$ and ends at $j$, namely  $w_{[i,j]} = w_{i} \dots w_{j}$. A prefix of $w$ is a subword of the form $w_{[0,i]}$, with $i<|w|$. Symmetrically, a suffix of a finite word $w$ is a subword of the form $w_{[i,|w|-1]}$ with  $i\geq 0$.

We put 
\begin{itemize}
\item $\alphab{A}^{n}$ for the set of the words of length $n$ of $\alphab{A}$, \item $\alphab{A}^{\star}$ for the set of the finite words of  $\alphab{A}$, 
\item $\alphab{A}^{\N}$ for the set of the infinite words of  $\alphab{A}$,
\item $\alphab{A}^{\Z}$ for the set of the bi-infinite words of $\alphab{A}$,
\item $\words{\alphab A}=\alphab A^\star\sqcup\alphab A^\N\sqcup\alphab A^\Z$ for the set of all words.
\end{itemize}

A \defi{word set} on $\alphab{A}$ is a set $X\subset\words{\alphab A}$ which contains a finite or infinite number of words of any kind (i.e. finite, infinite or bi-infinite) on $\alphab{A}$. A \defi{language} is a word set which contains only finite words. 
A language $X$ on $\alphab{A}$ is \defi{prolongeable} if for all words $w\in X$, there exist two letters $a$ and $b$ in $\alphab{A}$ such that $awb\in X$. For all positive integers $N$, $X$ is \defi{$N$-prolongeable} if for all words $w$ of length $N$ in $X$, there exist a letter $a\in\alphab{A}$ such that $aw\in X$ and a letter $b\in\alphab{A}$ such that $wb\in X$. If a language is prolongeable then it is $N$-prolongeable for all positive integers $N$.

For all integers $n$, the set of the subwords of length $n$ of a word set $X$ is noted $\Nlan[n]{X}$. We put $\lan{X}$ for $\bigcup_{n=1}^{\infty} \Nlan[n]{X}$.

Let  $\alphab{A}$ and $\alphab{B}$ be two alphabets. Maps from $\alphab{A}$ to $\alphab{B}$ are called projections and can be extended by concatenation to maps from $\words{\alphab{A}}$ to $\words{\alphab{B}}$.

Let $X$ and $Y$ be two word sets. We write $X \succcurlyeq Y$  if there exists a projection $\varphi$ such that $\varphi(X) = Y$.
If we have both $X \succcurlyeq Y$ and $Y \succcurlyeq X$ then $X$ and $Y$ are said \defi{similar} (i.e. they are equal up to renaming their letters). We then write $X\similar Y$.

\subsection{Subshifts}
The \defi{shift map} $\sigma$ is the bijective map of $\alphab{A}^{\Z}$ to itself which shifts all the sequences to the left. Namely, for all $u\in \alphab{A}^{\Z}$, we have:
\begin{dmath*}
\sigma(u)_{n} = u_{n+1} \condition{for all $n \in \Z$.}
\end{dmath*}

A \defi{subshift} is a subset $X$ of $\alphab{A}^{\Z}$, thus a word set, which is both topologically closed and shift invariant, \ie such that $\sigma(X) = X$.
If $X$ is a subshift, then there exists a set $F \subseteq \alphab{A}^{\star}$ such that for every $u \in \alphab{A}^{\Z}$, the word $u$ belongs to $X$ if, and only if, $\mathcal{L}(u)\cap F = \emptyset$ \cite{Lind}. The set $F$ is called a \defi{forbidden language} for $X$.
Note that, since a subshift $X$ contains only bi-infinite words, its language $\lan{X}$ is always prolongeable. 

A subshift $X$ is said of \defi{finite type} (SFT) if it admits a finite forbidden language. In this case, one can assume without loss of of generality that all the words of the forbidden language have a same length \cite{Lind}. If they can be assumed to have length $\Stp+1$, we say that the SFT is \defi{$\Stp$-step}. Note that it is then also $N$-step for all $N\geq\Stp$.

If an $\Stp_{X}$-step SFT $X$ and an $\Stp_{Y}$-step SFT $Y$ are similar, then they are $\min\{\Stp_{X},\Stp_{Y}\}$-step SFTs.

\subsection{$N$-block presentations} 
Let $\alphab{A}$ be an alphabet and $N$ an integer. By defining the $N$-block alphabet $\toalph{\alphab{A}^{N}}$ as $\toalph{\alphab{A}^{N}} = \{[w]\condi w\in\alphab{A}^{N}\}$, the \defi{$N^{\mbox{\tiny th}}$ higher-block code} $\Phi_{N}$ is the map from $\words{\alphab A}$ to $\words{\toalph{\alphab{A}^{N}}}$ defined by:
\begin{dmath*}
(\Phi_{N}(u))_{i} = \left[u_{[i,i+N-1]}\right],
\end{dmath*}
for all words $u\in \words{\alphab A}$ and all positions $i$ of $u$ such that $i+N-1$ is still a position of  $u$. We use the notation $[w]$ to avoid confusion between the finite word $w$ and the corresponding letter $[w]$ of the block alphabet.

The \defi{$N$-block presentation} of $X$ is $\NBP{X}{N} = \Phi_{N}(X)$ which is a word set over $\Nlan{X}$. By abuse, we will say that a word set $Y$ is the $N$-block presentation of a word set $X$ if $Y$ is similar to $\NBP{X}{N}$.

The $N$-block presentation of a word set $X$ is well defined if $X$ contains only words of length greater or equal to $N$, a property which is assumed granted for all the word sets considered from now on.

For instance, the $3$-block presentation of $V \hiderel{=} \{b a b e c b a b a b e c b e d e d e c b\}$ is
\setlength{\tabcolsep}{-1pt}
\begin{center}
\begin{tabular}{lcccccccccccccccccccccc}
$\NBP{V}{3}\;$ & $=\;$ & $\bigg\{$ & $\colt{b}{a}{b}$ & $\colt{a}{b}{e}$ & $\colt{b}{e}{c}$ & $\colt{e}{c}{b}$ & $\colt{c}{b}{a}$ & $\colt{b}{a}{b}$ & $\colt{a}{b}{a}$ & $\colt{b}{a}{b}$ & $\colt{a}{b}{e}$ & $\colt{b}{e}{c}$ & $\colt{e}{c}{b}$ & $\colt{c}{b}{e}$ & $\colt{b}{e}{d}$ & $\colt{e}{d}{e}$ & $\colt{d}{e}{d}$ & $\colt{e}{d}{e} $ & $\colt{d}{e}{c}$ & $\colt{e}{c}{b}$ & $\bigg\}$\\[0.5cm]
&$\similar\;$ & $\{$ & $1$ & $0$ & $5$ & $9$ & $2$ & $1$ & $3$ & $1$ & $0$ & $5$ & $9$ & $4$ & $7$ & $\dix$ & $6$ & $\dix$ & $ 8$ & $9$ & $\}$
\end{tabular}
\end{center}

\begin{rem}\label{r:compo}
For all positive integers $M$ and $N$, the $M$-block presentation of the $N$-block presentation of a word set $X$ is similar to its $(N+M-1)$-block presentation: 
\begin{dmath*}
\NBP{(\NBP{X}{N})}{M} \similar \NBP{X}{N + M - 1}.
\end{dmath*}
\end{rem}

\begin{rem}[\cite{Lind}]\label{r:1step}
For all integers $N,L\ge1$, a subshift $X$ is an $L$-step SFT if and only if $\NBP{X}{N}$ is a $\max(1,L-N+1)$-step SFT.
\end{rem}
In particular, note that the finite-type property is preserved by higher-block presentation.

\begin{rem}\label{r:period}
Let $X$ be a subshift. If there exists a positive integer $N$ such that $\card{\Nlan[N]{X}} = \card{\Nlan[N+1]{X}}$ then $\NBP{X}{N+K}\similar \NBP{X}{N}$ for all integers $K\geq 0$ (actually $X$ is periodic, i.e. is made of finitely many periodic words).
\end{rem}



\section{Characterization and preimages of $N$-block presentations} \label{s:preimage}

In order to deal with word sets, we have to adapt the  definition of the equivalence relations used to characterize $N$-block presentations of single words in \cite{ecriture}. 

Let $X$ be a word set on an alphabet $\alphab{A}$. For all $k \in \{0,1, \ldots, N-1\}$, the relation $\nrelaeq{k}$ is defined for all symbols $a$ and $b$ by $a\nrelaeq{k}b$ if there exist a non-negative integer $n$, a sequence $\x{0}, \x{1}, \ldots, \x{n}$ of letters of $\alphab{A}$ and a sequence $\y{1}, \y{2}, \ldots, \y{n}$ of elements in $\{-1, 1\}$, which are such that
\begin{itemize}
\item $\x{0} = a$, $\x{n} = b$ and $\displaystyle\sum_{j=1}^{n}\y{j} = 0$,

\item for all integers $0< i\leq n$, we have
\begin{dmath*}
-k\hiderel{\leq}\sum_{j=1}^{i}\y{j}\hiderel{<}N-k  \mbox{ and }\left\{\begin{array}{llr}
\x{i-1}\x{i}\in\Nlan[2]{X} & \mbox{ if } \y{i} = & -1,\\
\x{i}\x{i-1}\in\Nlan[2]{X} & \mbox{ if } \y{i} = & 1.
\end{array}\right. 
\end{dmath*}
\end{itemize}

For all integers $0\leq k<N$, the relation $\nrelaeq{k}$ is an equivalence relation. We put $\parp{k}{N}$ for the corresponding partition of $\alphab{A}$. 

\begin{theo}[\cite{ecriture}]\label{t:carac}
Let $X$ be a word set and $\alphab A=\Nlan[1]X$.
$X$ is an $N$-block presentation if and only if for all pairs of symbols $(a,b)\in \alphab A$ with $a\neq b$, there exists an integer $0\leq k<N$ such that $a$ and $b$ are not in relation with $\nrelaeq{k}$. 
\end{theo}

\begin{proof}
Let us first assume that $X$ is similar to an $N$-block presentation. All letters $a\in\alphab{A}$ are associated in a one-to-one way with a subword/block of length $N$, which we write $[a_{0}\ldots a_{N-1}]$. If we have $ab\in \Nlan[2]{X}$ then the corresponding blocks overlap. Namely we have $a_{i}=b_{i-1}$ for all $0<i<N$.
Let us assume that $a\nrelaeq{k}b$. There exists an integer $n$ and two sequences $\x{0}, \x{1}, \ldots, \x{n}$ and $\y{1}, \y{2}, \ldots, \y{n}$ which satisfy the definition above. It is straightforward to prove by induction that, for all $0< i\leq n$, 
\begin{dmath*}
\x{i}_{k+\sum_{j=1}^{i}\y{j}}=\x{0}_{k}.
\end{dmath*}
In particular, we get that $a\nrelaeq{k}b$ implies $a_{k} = b_{k}$. It follows that if $X$ is an $N$-block presentation and $a\neq b$, there exists at least an integer $0\leq k<N$ such that $a$ and $b$ are not in relation with $\nrelaeq{k}$.

Reciprocally, let us assume that for all symbols $a\neq b$ of $\alphab{A}$, there exists at least an integer $0\leq k<N$ such that $a$ and $b$ are not in relation with $\nrelaeq{k}$.
Let us define, for all $D\subset\alphab{A}^{2}$ and all $p\subset\alphab{A}$, 
$\slr{D}{p} = \{b\condi\exists a\in p, ab\in D\}$ and 
$\sll{D}{p} = \{b\condi\exists a\in p, ba\in D\}$.
An argument similar to that of \cite[Lemma 2]{ecriture} shows that, for all integers $0\leq k< N-1$ and all classes $p$ of $\parp{k}{N}$, if $\slr{\Nlan[2]{X}}{p}$  is not empty, there exists a unique class $q \in \parp{k+1}{N}$  such that $\slr{\Nlan[2]{X}}{p}\subset q$. Symmetrically, for all integers  all $0< k\leq N-1$ and all classes $p$ of $\parp{k}{N}$, if $\sll{\Nlan[2]{X}}{p}$ is not empty, there exists a unique class  $q \in \parp{k-1}{N}$ such that  $\sll{\Nlan[2]{X}}{p} \subset q$.

For all $0\leq\ell<N$, we define the alphabet $\alphab{B}_{\ell}$ as the set of $N$-uples $(p_{0}, p_{1}, \ldots, p_{N}) \in (\parp{0}{N}\cup\{\emptyset\})\times\ldots\times(\parp{N-1}{N}\cup\{\emptyset\})$ which are such that 
\begin{itemize}
\item $p_{\ell}\in\parp{\ell}{N}$, 
\item for all $0<i\leq \ell$,
\begin{dmath*}
p_{i-1} = \left\{\begin{array}{ll} 
\emptyset & \mbox{ if } p_{i}= \emptyset \mbox{ or }\sll{\Nlan[2]{X}}{p_{i}}=\emptyset,\\ 
\mbox{$q\in\parp{i-1}{N}$ s.t. $\sll{\Nlan[2]{X}}{p_{i}}\subset q$} &\mbox{otherwise,}
\end{array}\right.
\end{dmath*}
\item for all $\ell<i<N-1$,
\begin{dmath*}
p_{i+1} = \left\{\begin{array}{ll} 
\emptyset & \mbox{ if } p_{i}= \emptyset \mbox{ or }\slr{\Nlan[2]{X}}{p_{i}}=\emptyset,\\ 
\mbox{$q\in\parp{i+1}{N}$ s.t. $\slr{\Nlan[2]{X}}{p_{i}}\subset q$} &\mbox{otherwise.}
\end{array}\right.
\end{dmath*}
\end{itemize}
By construction, if there exists $0\leq i<j<N$ and $0\leq m<N$ such that $(p_{0}, \ldots, p_{N})\in\alphab{B}_{i}$, $(p'_{0}, \ldots, p'_{N})\in\alphab{B}_{j}$ and $p_{m} = p'_{m}$ then $p_{n} = p'_{n}\neq\emptyset$ for all $0\leq n<N$.
By setting $\alphab{B} = \bigcup_{i=0}^{N-1}\alphab{B}_{i}$, for all $a\in\alphab{A}$ and all $0< k\leq N-1$ there exists only one element $(p_{0}, \ldots, p_{N})\in\alphab{B}$ such that $a\in p_{k}$.
For all integers $0< k\leq N-1$, we define $\proj{k}$ as the letter-to-letter map which associates all symbols $a\in\alphab{A}$ with the unique element $(p_{0},\ldots, p_{N-1})\in\alphab{B}$ such that $a\in p_{k}$.
Under the current assumption, the map which associates to all letters $a\in\alphab{A}$, the $N$-block $[\proj{0}(a)\ldots\proj{N-1}(a)]$ is one-to-one.
Let us define the transformation $\projI$ from $X$ to $\words{\alphab A}$ by:
\begin{dmath*}
\projI(u) = \left\{\begin{array}{ll}
\proj{0}(u)\proj{1}(u_{|u|-1})\ldots\proj{N-1}(u_{|u|-1}) & \mbox{if $u\in\alphab{A}^{\star}$,}\\
\proj{0}(u) & \mbox{otherwise.}
\end{array}\right.
\end{dmath*}
By construction, for all $a$ and $b$ in $\alphab{A}$, if $ab\in\Nlan[2]{X}$ then $\proj{i}(a)=\proj{i-1}(b)$ for all $0<i<N$. At all positions $\ell$ of all words $u\in X$, we have that $(\projI(u))_{\ell+i} = \proj{i}(u_{\ell})$ for all $0\leq i<N$. It follows that $X$ is similar to $\NBP{(\projI(X))}{N}$.
\end{proof}
We emphasize that deciding if a given word set is a  $N$-block presentation only relies on its set of subwords of length $2$. In the case where a word set  is actually an $N$-block presentation, its set of subwords of length $2$ suffices for determining a projection leading to one of its preimages.

\begin{defn} Let $X$ be a word set and $N$ a non negative integer. Any word set $Y$ such that $\NBP{Y}{N}$ is similar to $X$ is an $N$-preimage of $X$. 
\end{defn}

\begin{cor}[\cite{ecriture}]\label{c:maxpre}
Let $X$ and $Y$ be two word sets. If $Y$ is an $N$-preimage of $X$, then 
\begin{dmath*}
Y \preccurlyeq \projI(X),
\end{dmath*}
where $\projI$ is the transformation defined in the proof of Theorem \ref{t:carac}.

In other words, there exists a projection $\psi$ from $\Nlan[1]{\projI(X)}$ to $\Nlan[1]{Y}$ such that $\psi(\projI(X)) = Y$.
\end{cor}
Corollary \ref{c:maxpre} is proved by using  an argument similar to that of \cite[Corollary 2]{ecriture}, which essentially relies on the second part of the proof of Theorem \ref{t:carac}.

This corollary ensures that any word set with an $N$-block presentation similar to $X$ can be obtained by projection from $\projI(X)$ or from any word set similar to $\projI(X)$. We will (improperly) refer to any word set similar to $\projI(X)$ as ``the'' \defi{maximal $N$-preimage} of $X$, which will be noted $\NBP{X}{-N}$.

\begin{prop}\label{p:proj}
Let $N$ be a positive integer and $X$ and $Y$ be two word sets which are $N$-block presentations. If $X\succcurlyeq Y$ then $\NBP{X}{-N}\succcurlyeq \NBP{Y}{-N}$.
\end{prop}
\begin{proof}
Let us assume that there is a projection $\projZ$ such that $\projZ(X) = Y$. Since if $cd\in\Nlan[2]{X}$ then $\projZ(c)\projZ(d)\in\Nlan[2]{Y}$, for all letters $a$ and $b$ in $\Nlan[1]{X}$ and for all $0\leq k<N$, if $a \nrelaeq{k} b$ then $\projZ(a) \nrelaeq[Y]{k} \projZ(b)$. Under the notations of the proof of Theorem \ref{t:carac}, we have that, for all letters $a$ and $b$ in $\Nlan[1]{X}$ and for all $0\leq k<N$, if $\proj{k}(a)=\proj{k}(b)$ then $\projBis{k}(\projZ(a))=\projBis{k}(\projZ(b))$. The projection $\projZ'$ from $\Nlan[1]{\NBP{X}{-N}}$ to $\Nlan[1]{\NBP{Y}{-N}}$ which associates to all letters $c\in\Nlan[1]{\NBP{X}{-N}}$, $\projZ'(c) = \proj{k}(\projZ(a))$ for any $a\in(\proj{k})^{-1}(c)$, is thus well defined and such that $\projZ'(\NBP{X}{-N}) =\NBP{Y}{-N}$.
\end{proof}

\section{Composing maximal $N$-preimages}\label{s:composing}
Remark \ref{r:compo} states that, for all positive integers $M$ and $N$, the $M$-block presentation of the $N$-block presentation of a word set is similar to its $(N+M-1)$-block presentation. We shall see that the situation is not that simple for the $N$-preimages.

\begin{rem}\label{p:remonte}
Let $N>1$ and $M>1$ be two integers and $X$ be a word set which is a  $(N + M - 1)$-block presentation. The maximal $N$-preimage of $X$ is not always an $M$-block presentation.
\end{rem}

In order to illustrate Remark \ref{p:remonte}, let us consider the example where $M=N=2$ and $X  = \{1 0 5 9 2 1 3 1 0 5 9 4 7 \dix 6 \dix  8 9\} \similar \NBP{V}{3}$ with $V \hiderel{=} \{b a b e c b a b a b e c b e d e d e c b\}$.

We have
\begin{dmath*}
\Nlan[2]{X} = \{10, 05, 59, 92, 21, 13, 31, 94, 47, 7\dix, \dix 6, 6\dix, \dix 8, 89 \},
\end{dmath*}
from which we get
\begin{dgroup*}
\begin{dmath*}
\parp[X]{0}{2} = \{\{0, 3\}, \{2, 4\}, \{6, 8\}, \{1\}, \{5\}, \{7\}, \{9\},\{\dix\} \},
\end{dmath*}
\begin{dmath*}
\parp[X]{1}{2} = \{\{5, 8\}, \{6, 7\}, \{2, 3\}, \{0\}, \{1\}, \{4\}, \{9\},\{\dix\} \}.
\end{dmath*}
\end{dgroup*}
By setting
\begin{dmath*}
 \begin{array}{cclr}
   & & \parp[X]{0}{2} & \parp[X]{1}{2} \\
 A & = & (\{0, 3\},  &  \{1\}) \\
 B & = & (\{2, 4\}, &  \{9\}) \\
 C & = & (\{6, 8\}, &   \{\dix\}) \\
 D & = & (\{1\}, & \{2, 3\}) \\
 E & = & (\{5\}, &  \{0\}) \\
 F & = & (\{7\}, &   \{4\} ) \\
 G & = & (\{9\}, &  \{5, 8\} ) \\
 H & = & (\{\dix\}, &  \{6, 7\}   )\\
\end{array}
\end{dmath*}

the maximal $2$-preimage of $X$ is
\begin{dmath*}
Y = \{D A E G B D A D A E G B F H C H C G B\}.
\end{dmath*}

Let us check if $Y$ is itself a $2$-block presentation.
We have
\begin{dmath*}
\Nlan[2]{Y} = \{DA, AE, EG, GB, BD, AD, BF, FH, HC, CH, CG \}
\end{dmath*}
thus
\begin{dgroup*}
\begin{dmath*}
\parp[Y]{0}{2} = \{\{{E}, D, {F}\}, \{H, G\}, \{A\}, \{B\}, \{C\} \}
\end{dmath*}
\begin{dmath*}
\parp[Y]{1}{2} = \{\{ {E}, C, {F}\}, \{A, B\}, \{D\}, \{G\}, \{H\} \}.
\end{dmath*}
\end{dgroup*}
Since the letters $E$ and $F$ belong to the same class in both partitions $\parp[Y]{0}{2}$ and $\parp[Y]{1}{2}$, Theorem \ref{t:carac} ensures that $Y$ is not a $2$-block presentation. 

\begin{prop}
Let $N$ and  $M > 1$ be two positive integers, $X$ be an $(N + M -1)$-block presentation, $Y = \NBP{X}{-N-M+1}$ its maximal $(N + M - 1)$-preimage and $Z = \NBP{X}{-N}$ its maximal $N$-preimage.
If $Z$ is an $M$-block presentation then $Y$ is its maximal $M$-preimage.
 \end{prop}
\begin{proof}
From Remark \ref{r:compo}, the word set $X$ is an $N$-block presentation. 
Since $\NBP{(\NBP{Y}{M})}{N}\similar X$ and $Z$ is the maximal $N$-preimage of $X$, Corollary \ref{c:maxpre} ensures that $Z\succcurlyeq \NBP{Y}{M}$. If moreover $Z$ is an $M$-block presentation then Proposition \ref{p:proj} gives us that $\NBP{Z}{-M}\succcurlyeq Y$.

On the other hand, we have that $\NBP{(\NBP{Z}{-M})}{N+M-1}\similar X$ (Remark \ref{r:compo}). Since $Y$ is the maximal $(N+M-1)$-preimage of $X$, Corollary \ref{c:maxpre} gives us that $Y\succcurlyeq \NBP{Z}{-M}$ which ends the proof.
\end{proof}

\section{Characterizing  maximal $N$-preimages}\label{s:charac}
\begin{defn}Let $X$ be a word set. For all $a\in\alphab{A}$ and all integers $N>1$, the  graph of order $N$ of $a$ with regard to $X$ is the undirected graph $\ngra{a}{N} = (\mathcal{V}, \mathcal{E})$ where
\begin{itemize}
\item $\mathcal{V} = \{(i,u) \condi 0 \leq i < N, u \in \Nlan[N]{X}, u_{i} = a\}$. In plain English, vertices of $\ngra{a}{N}$ are pointed words of length $N$ which occur in $X$, with the letter $a$ at the pointed position;
\item $\mathcal{E} = \left\{\{(i,w_{[0,N-1]}), (i-1, w_{[1,N]})\} \condi 0 < i < N, w \in \Nlan[N+1]{X},w_{i} = a\right\}$. Edges are of the form $\{(i,u),(i-1,v)\}$, where $u$ and $v$ are respectively prefix and suffix of a word of $\Nlan[N+1]{X}$ in which $a$ occurs at the $i^{\mbox{\tiny th}}$ position.
\end{itemize}
\end{defn}

By construction, the graph $\ngra{a}{N}$ is $N$-partite (Figure \ref{f:fig1}).

\begin{figure}[!hb]
\centering{\includegraphics{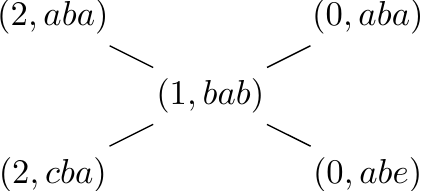}}
\caption{The graph $\ngra[V]{a}{3}$ with $V = \{b a b e c b a b a b e c b e d e d e c b\}$.}
\label{f:fig1}
\end{figure}

A letter $a\in\Nlan[1]{X}$ is \defi{$N$-connected} with regard to  $X$ if the graph $\ngra{a}{N}$ is connected (in the usual sense). For instance, $a$ is $3$-connected with regard to $V = \{b a b e c b a b a b e c b e d e d e c b\}$ (Figure \ref{f:fig1}). \\

\begin{lemma}\label{l:graph}
Let $X$ be a word set. For all letters $a\in\Nlan[1]{X}$, we have that
\begin{enumerate}
\item $\ngra{a}{1}$ is connected;
\item if $\ngra{a}{N}$ is connected then, for all positive integers $M\leq N$, $\ngra{a}{M}$ is connected.
\end{enumerate}
\end{lemma}
\begin{proof}
The first assertion is plain since $\ngra{a}{1}$ contains only the vertex $(0,a)$.

In order to prove Assertion 2, let us first remark that if there is an edge $\{(i, u),(i-1, v)\}$ in $\ngra{a}{N}$ then, by setting $t=uv_{N-1}$, for all $\max\{i-M-1, 0\}\leq\ell\leq\min\{i,N-M\}$, $(i-\ell, t_{[\ell,\ell+M-1]})$ and $(i-\ell-1, t_{[\ell+1,\ell+M]})$ are adjacent vertices of $\ngra{a}{M}$.  Reciprocally, since we made the implicit assumption that all words of $X$ are longer than $N$, for all pairs of vertices $(j, w)$ and $(j', w')$ of $\ngra{a}{M}$, there exists two vertices $(i, u)$ and  $(i', u')$  of $\ngra{a}{N}$ which are such that $u_{[i-j, i-j+M-1]} = w$ and $u'_{[i'-j', i'-j'+M-1]} = w'$. If the graph  $\ngra{a}{N}$ is connected then there exists a  path between $(i, u)$ and $(i', u')$ which then implies the existence of a path between $(j, w)$ and $(j', w')$ in  $\ngra{a}{M}$.
\end{proof}

\begin{lemma}\label{l:graphbis}
Let $X$ be a word set and $a$ be a letter in $\Nlan[1]{X}$. If there exists an integer $N\geq 2$ such that 
\begin{itemize}
\item $\ngra{a}{N}$ is not connected,
\item $\ngra{a}{N-1}$ is connected and
\item $\lan{X}$ is $(N-1)$-prolongeable,
\end{itemize}
then there exists a word $w\in\Nlan[1]{X}^{N+1}$ such that $w_{[0,N-1]}\in\Nlan[N]{X}$, $w_{[1,N]}\in\Nlan[N]{X}$ and $w\not\in\Nlan[N+1]{X}$.
\end{lemma}
\begin{proof}
Let us assume that $\ngra{a}{N}$ is not connected. It contains two vertices $(i, u)\neq(j, v)$ which are not connected. 
Let us set 
\begin{dmath*}
p \hiderel{=} \left\{\begin{array}{ll} (i,u_{[0,N-2]})& \mbox{if }i<N-1\\ (i-1,u_{[1,N-1]}) & \mbox{otherwise,}\end{array}\right. \mbox{ and } q \hiderel{=} \left\{\begin{array}{ll} (j-1,v_{[1,N-1]})& \mbox{if }j>0\\ (j,v_{[0,N-2]}) & \mbox{otherwise.}\end{array}\right.\end{dmath*}
Vertices $p$ and $q$ are both in $\ngra{a}{N-1}$.
By assuming $\ngra{a}{N-1}$ is connected, there exists a sequence of vertices $(\z{0}, \w{0}), \ldots, (\z{n}, \w{n})$ of $\ngra{a}{N-1}$, with $(\z{0},\w{0}) = p$, $(\z{n},\w{n}) = q$ and such that, for all $0< m \leq n$,  $(\z{m-1},\w{m-1})$ and  $(\z{m},\w{m})$ are adjacent in $\ngra{a}{N-1}$, i.e. either $\z{m}=\z{m-1}-1$ and $\w{m-1}\w{m}_{N-2}=\w{m-1}_{0}\w{m}\in\Nlan[N]{X}$ or $\z{m}=\z{m-1}+1$ and $\w{m}\w{m-1}_{N-2}=\w{m}_{0}\w{m-1}\in\Nlan[N]{X}$. Let us set:
\begin{itemize}
\item for all $0<m\leq n$, 
\begin{dmath*}
(\s{m},\y{m})=\left\{\begin{array}{lrl} (\z{m-1},&\w{m-1}\w{m}_{N-2}) & \mbox{if } \z{m}=\z{m-1}-1,\\ (\z{m}, &\w{m}\w{m-1}_{N-2})& \mbox{if } \z{m}=\z{m-1}+1,\end{array}\right.
\end{dmath*}
\item $o=\left\{\begin{array}{ll}0 \mbox{ with } (\s{0}, \y{0})=(i,u) & \mbox{if }i=N-1, \\ 1 & \mbox{otherwise,}\end{array}\right.$
\item $r=\left\{\begin{array}{ll}n+1 \mbox{ with } (\s{n+1}, \y{n+1})=(j,v) & \mbox{if }j=0, \\ n & \mbox{otherwise.}\end{array}\right.$ 
\end{itemize}
 We have that $r-o\geq 1$ and that $(\s{m}, \y{m})$ is a vertex of $\ngra{a}{N}$, for all integers  $o\leq m \leq r$.
 
Since $(\s{o},\y{o}) = (i, u)$ and $(\s{r},\y{r}) = (j, v)$ are not connected in $\ngra{a}{N}$, there exists an integer $o\leq m < r$ such that $(\s{m}, \y{m})$ and $(\s{m+1}, \y{m+1})$ are not connected. Four different cases arise:
\begin{enumerate}
\item  if $m=0$ or $\left\{\begin{array}{ll}\z{m}=\z{m-1}-1,\\ \z{m+1}=\z{m}-1,\end{array}\right.$ then $\left\{\begin{array}{ll}\s{m+1} =\s{m}-1,\\ \y{m+1}_{[0,N-2]} = \y{m}_{[1,N-1]};\end{array}\right.$
\item  if $m=n$ or $\left\{\begin{array}{ll}\z{m}=\z{m-1}+1,\\ \z{m+1}=\z{m}+1,\end{array}\right.$ then $\left\{\begin{array}{ll}\s{m+1} =\s{m}+1,\\ \y{m+1}_{[1,N-1]} = \y{m}_{[0,N-2]};\end{array}\right.$
\item  if $\left\{\begin{array}{ll}\z{m}=\z{m-1}+1,\\ \z{m+1}=\z{m}-1,\end{array}\right.$ then $\left\{\begin{array}{ll}\s{m+1} =\s{m},\\ \y{m+1}_{[0,N-2]} = \y{m}_{[0,N-2]};\end{array}\right.$
\item  if $\left\{\begin{array}{ll}\z{m}=\z{m-1}-1,\\ \z{m+1}=\z{m}+1,\end{array}\right.$ then $\left\{\begin{array}{ll}\s{m+1} =\s{m},\\ \y{m+1}_{[1,N-1]} = \y{m}_{[1,N-1]}.\end{array}\right.$
\end{enumerate}
In Case 1 (resp. in Case 2), since the vertices $(\s{m},\y{m})$ and  $(\s{m+1},\y{m+1})$ are not connected, they are not adjacent, which arises only if $\y{m}\y{m+1}_{N-1}=\y{m}_{0}\y{m+1}\not\in\Nlan[N+1]{X}$ (resp. only if $\y{m+1}\y{m}_{N-1}=\y{m+1}_{0}\y{m}\not\in\Nlan[N+1]{X}$).

Let us remark that we have always $0<\s{m}=\s{m+1}<N-1$ in Cases 3 and 4.

In Case 3 and since $\lan{X}$ is $(N-1)$-prolongeable, there exists a letter $b\in\Nlan[1]{X}$ such that $b\y{m}_{[0,N-2]}=b\y{m+1}_{[0,N-2]}\in\Nlan[N]{X}$, which implies that $(\s{m}+1, b\y{m}_{[0,N-2]})$ is a vertex of $\ngra{a}{N}$. Since $(\s{m},\y{m})$ and  $(\s{m+1},\y{m+1})$ are not connected, the vertex $(\s{m}+1, b\y{m}_{[0,N-2]})$ cannot be adjacent to both $(\s{m},\y{m})$ and  $(\s{m+1},\y{m+1})$, which implies that $b\y{m}$ or  $b\y{m+1}$ does not belong to $\Nlan[N+1]{X}$.

Symmetrically in Case 4, since $(\s{m},\y{m})$ and  $(\s{m+1},\y{m+1})$ are not connected, there 
exists a letter $b\in\Nlan[1]{X}$ such that $\y{m}_{[1,N-1]}b=\y{m+1}_{[1,N-1]}b\in\Nlan[N]{X}$ and 
$\y{m}b$ or  $\y{m+1}b$ does not belong to $\Nlan[N+1]{X}$, which ends the proof.
\end{proof}

\begin{rem}\label{r:subgraph}
Let $X$ be a word set and $N$ be a positive integer. For all words $w\in\Nlan[N]{X}$, all positive integers $K$ and all integers $0\leq i<N$, the graph $\ngra[\NBP{X}{N}]{[w]}{K}$ is isomorphic to the subgraph of $\ngra[X]{w_{i}}{N+K-1}$ which contains all the vertices $(j,v)$ such that $v_{[j-i,j-i+N-1]} = w$ (Figure \ref{f:fig2}).
\end{rem}

\begin{figure}[!hb]
\centering{\includegraphics{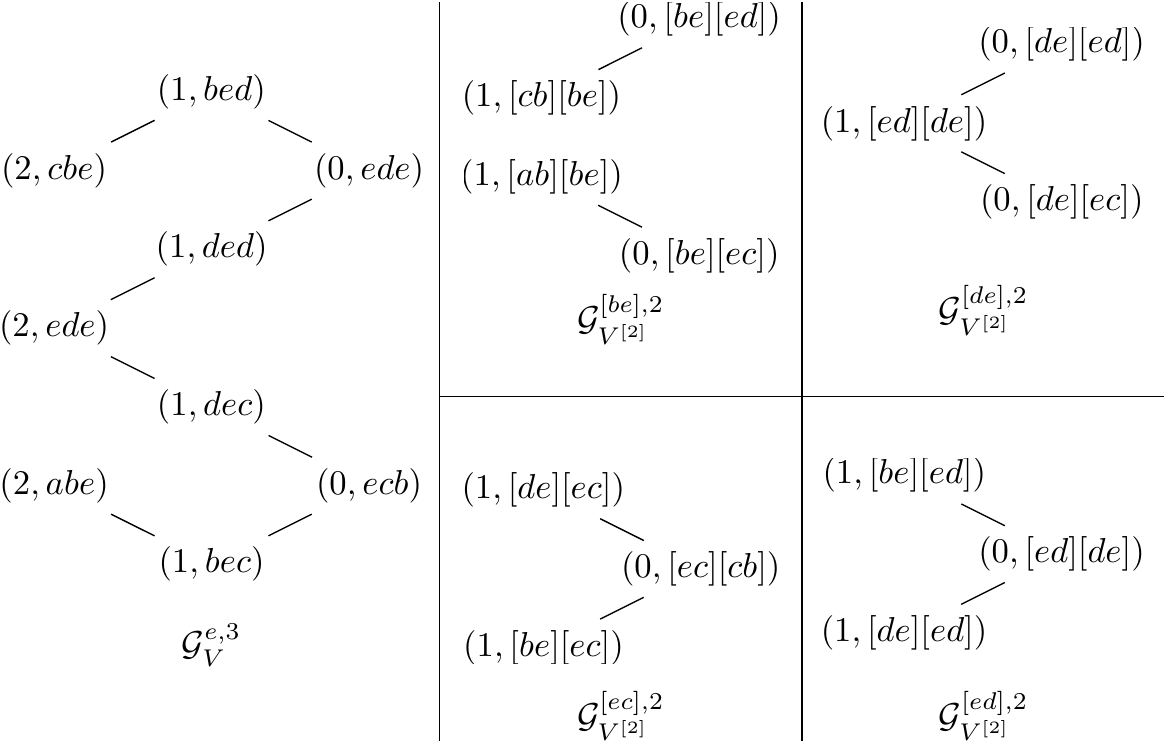}}
\caption{The graphs $\ngra[V]{e}{3}$, $\ngra[\NBP{V}{2}]{[be]}{2}$, $\ngra[\NBP{V}{2}]{[de]}{2}$, $\ngra[\NBP{V}{2}]{[ec]}{2}$ and  $\ngra[\NBP{V}{2}]{[ed]}{2}$ with $V = \{b a b e c b a b a b e c b e d e d e c b\}$. Note that $\ngra[\NBP{V}{2}]{[be]}{2}$ is not connected.}
\label{f:fig2}
\end{figure}

\begin{prop} \label{p:sommet_connecte} Let $X$ be a word set on an alphabet $\alphab{A}$, $N$ a positive integer, $u$ and $v$ be two words of $\Nlan{X}$ and $k$ be an integer with $0\leq k<N$. We have that $[u] \nrelaeq[\NBP{X}{N}]{k} [v]$ if and only if there exists a letter $a\in\alphab{A}$  such that $(k,u)$ and $(k,v)$ are two connected vertices of $\ngra{a}{N}$.
\end{prop}
\begin{proof}
 Let us first assume that  $(k,u)$ and $(k,v)$ are two connected vertices of $\ngra{a}{N}$. There exists a path $(\z{0}, \w{0}), \ldots, (\z{n}, \w{n})$ of $\ngra{a}{N}$ where 
\begin{itemize}
\item $(\z{0}, \w{0}) = (k,u)$ and $(\z{n}, \w{n}) = (k,v)$,
\item for all $0< i\leq n$, $\{(\z{i-1}, \w{i-1}), (\z{i}, \w{i})\}$ is an edge of $\ngra{a}{N}$.
\end{itemize}
By construction, for all $0< i\leq n$, the edge $\{(\z{i-1}, \w{i-1}), (\z{i}, \w{i})\}$ is such that either $[\w{i-1}][\w{i}]\in\Nlan[2]{\NBP{X}{N}}$ and $\z{i}=\z{i-1}-1$, or $[\w{i}][\w{i-1}]\in\Nlan[2]{\NBP{X}{N}}$ and $\z{i}=\z{i-1}+1$. By setting $\y{i}=\z{i}-\z{i-1}$ for all $0<i\leq n$, we get that $\sum_{j=1}^{n}\y{j}=\z{n}-\z{0}=0$. Moreover, since $\z{i} =\z{0}+\sum_{j=1}^{i}\y{j}$ and $0\leq \z{i}<N$, we have that $k\leq\sum_{j=1}^{i}\y{j}<N-k$.  Considering $\y{1}, \ldots \y{n}$ and $\x{i} = [\w{i}]$ for all $0\leq i\leq n$ leads to $[u]\nrelaeq{k}[v]$.

Reciprocally, if $[u]\nrelaeq{k}[v]$ then there exist a non-negative integer $n$, a sequence $[\x{0}], [\x{2}], \ldots, [\x{n}]$ of blocks of $\Nlan{X}$ and a sequence $\y{1}, \y{2}, \ldots, \y{n}$ of elements in $\{-1, 1\}$, verifying
\begin{itemize}
\item $\x{0} = u$, $\x{n} = v$ and $\displaystyle\sum_{j=1}^{n}\y{j} = 0$,
\item for all integers $0< i\leq n$, we have
\begin{dmath*}
-k\hiderel{\leq}\sum_{j=1}^{i}\y{j}\hiderel{<}N-k  \mbox{ and }\left\{\begin{array}{llr}
\x{i-1}\x{i}_{N-1}\in\Nlan[N+1]{X} & \mbox{ if } \y{i} = & -1,\\
\x{i}\x{i-1}_{N-1}\in\Nlan[N+1]{X} & \mbox{ if } \y{i} = & 1.
\end{array}\right. 
\end{dmath*}
\end{itemize}
For all integers $0\leq i\leq n$, we have $\x{i}_{k+\sum_{j=1}^{i}\y{j}}=\x{0}_{k}$. By setting $a=\x{0}_{k}$, all the pairs $(k+\sum_{j=1}^{i}\y{j}, \x{i})$ are vertices of $\ngra{a}{N}$ which are such that $\{(k+\sum_{j=1}^{i}\y{j}, \x{i}), (k+\sum_{j=1}^{i-1}\y{j}, \x{i-1})\}$ are edges of  $\ngra{a}{N}$. The vertices $(k,u)$ and $(k,v)$ are thus connected.
\end{proof}

\begin{cor}\label{c:numpre}Let $N>1$ be an integer, $X$ be a word set, $Y$ be the maximal $N$-preimage of $\NBP{X}{N}$ and $\psi$ be the projection  from $\Nlan[1]{Y}$ to $\Nlan[1]{X}$ such that $\psi(Y) = X$, whose existence is ensured by Corollary \ref{c:maxpre}.

For all $a\in\Nlan[1]{X}$, the number of letters in $\psi^{-1}(a)$ is equal to the number of  connected components of $\ngra{a}{N}$.\end{cor}

\begin{proof}
For all $0\leq k<N$, we define the projection $\projX{k}$ from $\Nlan{\NBP{X}{N}}$ to $\Nlan{Y}$, which associates to all letters $[w]=[w_{0}\ldots w_{N-1}]$, the letter $\projX{k}([w]) = w_{k}$.
By defining $\projXI$ on $\NBP{X}{N}$ as  $\projXI(\NBP{X}{N}) = \bigcup_{u\in \NBP{X}{N}} \projXI(u)$ where 
\begin{dmath*}
\projXI(u) = \left\{\begin{array}{ll}
\projX{0}(u)\projX{1}(u_{|u|-1})\ldots\projX{N-1}(u_{|u|-1}) & \mbox{if $u\in\alphab{A}^{\star}$,}\\
\projX{0}(u) & \mbox{otherwise,}
\end{array}\right.
\end{dmath*}
we have that $X = \projXI(\NBP{X}{N})$.

For all $a\in\Nlan[1]{X}$, a letter $[w] = [w_{0}\ldots w_{N-1}]$ of  $\NBP{X}{N}$ is such that $\projX{k}([w]) = a$ if and only if $(k, w)$ is a vertice of  $\ngra{a}{N}$.
Under the notation of the proof of Theorem \ref{t:carac}, for all $[w]\in\Nlan[1]{\NBP{X}{N}}$ and all $0\leq k<N$, we have that $\projX{k}([w]) = \psi(\proj{k}([w]))$.

Let $c$ be a letter in $\Nlan[1]{Y}$ such that $\psi(c) = a$. For all $[w]\in\Nlan[1]{\NBP{X}{N}}$ and for all $0\leq k<N$, if $\proj{k}([w]) = c$ then $(k, w)$ is a vertice of  $\ngra{a}{N}$.
From Proposition \ref{p:sommet_connecte}, it follows that there are exactly as much letters $c\in \Nlan[1]{Y}$ such that $\psi(c) = a$ as there are connected components in $\ngra{a}{N}$.
\end{proof}

\begin{cor}
Let $N$ be a positive integer. If a word set $X$ is a maximal $N$-preimage then it is a maximal $M$-preimage for all integers $1< M \leq N$.
\end{cor}
\begin{proof}
The result is a direct consequence of Lemma \ref{l:graph}-Item 2 and Corollary~\ref{c:numpre}.
\end{proof}

\begin{theo} 
A word set $X$ is the maximal $N$-preimage of its $N$-block presentation if and only if all letters in $\Nlan[1]X$ are $N$-connected.
\end{theo}
\begin{proof}
The theorem follows from  Corollary \ref{c:numpre}.
\end{proof}

We will say that a word set $X$ is ``a maximal $N$-preimage'' if it is the maximal $N$-preimage of its $N$-block presentation $\NBP{X}{N}$.

\begin{prop}\label{p:maxpre}
Let $N>K$ be two positive integers and $X$ be a word set. If for all subwords $w\in\Nlan[K]{X}$ the graphs $\ngra[\NBP{X}{K}]{[w]}{N-K+1}$ are connected then the maximal $(N+K-1)$-preimage of $X$ is a $K$-block presentation.
\end{prop}
\begin{proof}
Let us assume that, for all subwords $w\in\Nlan[K]{X}$, the graphs $\ngra[\NBP{X}{K}]{[w]}{N-K+1}$ are connected. From Corollary \ref{c:maxpre}, it follows that $\NBP{X}{K}$ is the maximal $(N-K+1)$-preimage of $\NBP{X}{N}$, which is thus itself a $K$-block presentation.
\end{proof}

Let's go back to the counter-example of Section \ref{s:composing} for which we observed that the maximal $2$-preimage of $\NBP{V}{3}$ is not a $2$-block presentation. Proposition \ref{p:maxpre} ensures that it can occur only if there exists a subword $w\in\Nlan[2]{V}$ such that the graph $\ngra[\NBP{V}{2}]{[w]}{2}$ is not connected. We do observe in Figure \ref{f:fig2} that the graph $\ngra[\NBP{V}{2}]{[be]}{2}$ is not connected.

\section{Direct conjugacy between SFTs}\label{s:decidability}

In this section, we focus on word sets which are SFTs, and we prove the decidability of a strong form of conjugacy.
\begin{defn}
Let $X$ and $Y$ be two SFTs. We say that $X$ and $Y$ are directly conjugate if there exists two positive integers $M$ and $N$ such that $\NBP{X}{M}$ and $\NBP{Y}{N}$ are similar.
\end{defn}
Direct conjugacy basically implies (topological) conjugacy.

\begin{rem}[\cite{Lind}]\label{r:Nblock}
Let $X$ be an SFT and $N$ a positive integer. A finite forbidden language for $\NBP{X}{N}$ can be computed in finite time from that of $X$.
\end{rem}

\begin{rem}
Being given an integer $N\ge0$ and an SFT $X$ represented by a forbidden language, the language $\Nlan[N]X$ can be computed in finite time.
\end{rem}

\begin{prop}\label{c:dec}
Checking if two SFTs are similar is decidable.
\end{prop}
\begin{proof}
We simply sketch the proof.
From any (one-dimensional) SFT, we can easily compute a set of minimal forbidden patterns.
Then it is not difficult to check all the alphabet bijections that send, when extended to words, each minimal forbidden pattern of the first one to one minimal forbidden pattern of the second one.
\end{proof}

\begin{lemma}\label{l:sftGraph}
Let $X$ be a SFT. If $X$ is $\Stp$-step, then for all positive integers $K$ and all words $w\in\Nlan[\Stp]{X}$, the graph $\ngra[\NBP{X}{\Stp}]{[w]}{K}$ is connected.
\end{lemma}
\begin{proof}
Let us assume that there exists an integer $K$ such that the graph $\ngra[\NBP{X}{\Stp}]{[w]}{K}$ is not connected, and that $K$ is the smallest such integer. From  Lemma \ref{l:graph}-Item 1, we have that $K\geq 2$.

Since the language $\lan{\NBP{X}{\Stp}}$ is $(K-1)$-prolongeable, Lemma \ref{l:graphbis} gives us that there exists a word $w\in \Nlan[1]{\NBP{X}{\Stp}}^{K+1}$ such that $w_{[0,K-1]}\in\Nlan[K]{\NBP{X}{\Stp}}$, $w_{[1,K]}\in\Nlan[K]{\NBP{X}{\Stp}}$ and $w_{[0,K]}\not\in\Nlan[K+1]{\NBP{X}{\Stp}}$. In other words, there is a minimal forbidden pattern of length $K+1\ge3$, which contradicts the fact that $\NBP{X}{\Stp}$ is $1$-step and thus that $X$ is an $\Stp$-step SFT (Remark \ref{r:1step}).
\end{proof}
\begin{lemma}\label{l:sftPreim}
Let $X$ be a SFT. If $X$ is $\Stp$-step then for all positive integers $K$, $\NBP{X}{\Stp}$ is the maximal $K$-preimage of $\NBP{X}{\Stp+K-1}$.
\end{lemma}
\begin{proof}
Result follows from Lemma \ref{l:sftGraph} and Proposition \ref{p:maxpre}.
\end{proof}

\begin{prop}
\label{p:conj}
Let $X$ and $Y$ be a $\Stp_{X}$- and a $\Stp_{Y}$-step SFTs, respectively.
We have that
\begin{enumerate}
\item if $\card{\Nlan[\Stp_{X}]{X}}<\card{\Nlan[\Stp_{Y}]{Y}}$, the SFTs $X$ and $Y$ are directly conjugate if and only if there exists an integer $K> 0$ such that $\NBP{X}{\Stp_{X}+K}$ is similar to $\NBP{Y}{\Stp_{Y}}$;
\item if $\card{\Nlan[\Stp_{X}]{X}}>\card{\Nlan[\Stp_{Y}]{Y}}$, the SFTs  $X$ and $Y$ are directly conjugate if and only if there exists an integer $K>0$ such that  $\NBP{X}{\Stp_{X}}$ is similar to $\NBP{Y}{\Stp_{Y}+K}$;
\item if $\card{\Nlan[\Stp_{X}]{X}}=\card{\Nlan[\Stp_{Y}]{Y}}$, the SFTs  $X$ and $Y$ are directly conjugate if and only if  $\NBP{X}{\Stp_{X}}$ is similar to $\NBP{Y}{\Stp_{Y}}$.
\end{enumerate}
\end{prop}
\begin{proof}
Let us assume that $X$ and $Y$ are directly conjugate with $\NBP{X}{M}\similar\NBP{Y}{N}$ for two positive integers $M$ and $N$. Since $\NBP{X}{M+J}\similar\NBP{Y}{N+J}$ for all $J\geq 0$, we assume without loss of generality that $M\geq \Stp_{X}$ and $N\geq \Stp_{Y}$. We distinguish between several cases according to the relative sizes of $\Nlan[\Stp_{X}]{X}$ and $\Nlan[\Stp_{Y}]{Y}$.
\begin{enumerate}

\item $\card{\Nlan[\Stp_{X}]{X}}<\card{\Nlan[\Stp_{Y}]{Y}}$. Let us first assume that $M-\Stp_{X}> N-\Stp_{Y}$. Since $\NBP{Y}{N}$ and $\NBP{X}{M}$ are similar, they have the same maximal $J$-preimages for all $0<J\leq\min\{M,N\}$. From Lemma \ref{l:sftPreim}, the maximal $(N-\Stp_{Y})$-preimage of $\NBP{Y}{N}$ is $\NBP{Y}{\Stp_{Y}}$ and, from Lemma \ref{l:sftGraph} and Proposition \ref{p:proj}, that of $\NBP{X}{M}$ is $\NBP{X}{M-N+\Stp_{Y}}$.
From Corollary \ref{c:maxpre}, we get that $\NBP{Y}{\Stp_{Y}}$ and $\NBP{X}{\Stp_{X}+K}$ with $K=M-N+\Stp_{Y}-\Stp_{X}> 0$, are similar.
Conversely, if $M-\Stp_{X}\leq N-\Stp_{Y}$, the maximal $(M-\Stp_{X})$-preimage of $\NBP{Y}{N}$ is $\NBP{Y}{N-M+\Stp_{X}}$ and that of $\NBP{X}{M}$ is $\NBP{X}{\Stp_{X}}$. We get that $\NBP{Y}{\Stp_{Y}+J}\similar\NBP{X}{\Stp_{X}}$ with $J=N-M+\Stp_{X}-\Stp_{Y}\geq 0$ and thus that $\card{\Nlan[\Stp_{Y}+J]{Y}}=\card{\Nlan[\Stp_{X}]{X}}$, which contradicts that $\card{\Nlan[\Stp_{X}]{X}}<\card{\Nlan[\Stp_{Y}]{Y}}$.

\item$\card{\Nlan[\Stp_{X}]{X}}>\card{\Nlan[\Stp_{Y}]{Y}}$. 
It is symmetrical with Case 1.

\item$\card{\Nlan[\Stp_{X}]{X}}=\card{\Nlan[\Stp_{Y}]{Y}}$. 
If $M-\Stp_{X}< N-\Stp_{Y}$ then considering  the maximal $(M-\Stp_{X})$-preimages of $\NBP{X}{M}$ and $\NBP{Y}{N}$ leads to $\NBP{X}{\Stp_{X}+K}\similar\NBP{Y}{\Stp_{Y}}$ with $K>0$. It implies that $\card{\Nlan[\Stp_{X}+K]{X}}=\card{\Nlan[\Stp_{X}]{X}}$, which itself implies that $\NBP{X}{\Stp_{X}}\similar\NBP{X}{\Stp_{X}+J}$  for all $J\geq 0$ (Remark \ref{r:period}) and in particular $\NBP{X}{\Stp_{X}}\similar\NBP{Y}{\Stp_{Y}}$. The case where $M-\Stp_{X}> N-\Stp_{Y}$ is symmetrical. If $M-\Stp_{X}=N-\Stp_{Y}$ then $\NBP{X}{M}\similar\NBP{Y}{N}$ implies that the maximal  $(M-\Stp_{X})$-preimages of $\NBP{X}{M}$ and $\NBP{Y}{N}$, which are $\NBP{X}{\Stp_{X}}$ and $\NBP{Y}{\Stp_{Y}}$, are similar.
\end{enumerate}

Conversely, if  $\card{\Nlan[\Stp_{X}]{X}}<\card{\Nlan[\Stp_{Y}]{Y}}$ and $\NBP{X}{\Stp_{X}+K}\similar\NBP{Y}{\Stp_{Y}}$ for a positive integer $K$, then $X$ and $Y$ are directly conjugate. The same holds if  $\card{\Nlan[\Stp_{X}]{X}}>\card{\Nlan[\Stp_{Y}]{Y}}$ and $\NBP{X}{\Stp_{X}}\similar\NBP{Y}{\Stp_{Y}+K}$ for a positive integer $K$, or if $\card{\Nlan[\Stp_{X}]{X}}=\card{\Nlan[\Stp_{Y}]{Y}}$ and $\NBP{X}{\Stp_{X}}\similar\NBP{Y}{\Stp_{Y}}$.
\end{proof}

\begin{theo}
The direct conjugacy between two given SFTs is decidable. 
\end{theo}

\begin{proof}
Let $X$ and $Y$ be two SFTs represented by two finite forbidden languages of patterns of lengths $\Stp_{X}+1$ and $\Stp_{Y}+1$ respectively. In other words, $X$ is $\Stp_{X}$-step and $Y$ is $\Stp_{Y}$-step.

If  $\card{\Nlan[\Stp_{X}]{X}}=\card{\Nlan[\Stp_{Y}]{Y}}$, Proposition \ref{p:conj} and Proposition \ref{c:dec} allow us to conclude.

Let us assume that $\card{\Nlan[\Stp_{X}]{X}}<\card{\Nlan[\Stp_{Y}]{Y}}$, the case where $\card{\Nlan[\Stp_{X}]{X}}>\card{\Nlan[\Stp_{Y}]{Y}}$ being symmetrical.
Proposition \ref{p:conj} ensures that $X$ and $Y$ are directly conjugate if and only if there exists a positive integer $K$ such that $\NBP{X}{\Stp_{X}+K}\similar\NBP{Y}{\Stp_{Y}}$, which implies that $\card{\Nlan[\Stp_{X}+K]{X}}=\card{\Nlan[\Stp_{Y}]{Y}}$ and suggests the following procedure to check if such an integer $K$ exists.
\begin{enumerate}
\item Initialize $J$ to $1$;
\item Compute $\Nlan[\Stp_{X}+J]{X}$;
\item If $\card{\Nlan[\Stp_{X}+J]{X}}=\card{\Nlan[\Stp_{X}+J-1]{X}}$ then we have that $\NBP{X}{\Stp_{X}+J+I}=\NBP{X}{\Stp_{X}+J}$ for all $I\geq 0$ and there is no integer $K$ such that  $\NBP{X}{\Stp_{X}+K}\similar\NBP{Y}{\Stp_{Y}}$, i.e. $X$ and $Y$ are not directly conjugate;
\item Otherwise, $\card{\Nlan[\Stp_{X}+J]{X}}>\card{\Nlan[\Stp_{X}+J-1]{X}}$ and 
\begin{itemize}
\item If $\card{\Nlan[\Stp_{X}+J]{X}}<\card{\Nlan[\Stp_{Y}]{Y}}$ then increment $J$ and go to step 2;
\item If $\card{\Nlan[\Stp_{X}+J]{X}}=\card{\Nlan[\Stp_{Y}]{Y}}$ then $X$ and $Y$ are directly conjugate if and only if $\NBP{X}{\Stp_{X}+J}\similar\NBP{Y}{\Stp_{Y}}$;
\item If $\card{\Nlan[\Stp_{X}+J]{X}}>\card{\Nlan[\Stp_{Y}]{Y}}$  then there is no integer $K$ such that  $\NBP{X}{\Stp_{X}+K}\similar\NBP{Y}{\Stp_{Y}}$, i.e. $X$ and $Y$ are not directly conjugate.
\end{itemize}
\end{enumerate}
The procedure above ends at most after $\card{\Nlan[\Stp_{Y}]{Y}}-\card{\Nlan[\Stp_{X}]{X}}$ iterations in which the only operations are computing the languages of a given order and the block presentations of SFTs or testing the similarity between SFTs. It thus can be performed in finite time (Remark \ref{r:Nblock} and Proposition \ref{c:dec}).
\end{proof}

\bibliographystyle{abbrv}
\bibliography{biblio}

\end{document}